\documentclass[10pt]{book}
\usepackage{array}
\usepackage{amssymb}
\RequirePackage{amsthm,amsmath}

\usepackage{color,latexsym,amsfonts,amssymb}

\newcommand{\ignore}[1]{}{}

\newcommand{\bea}{\begin{eqnarray}}
  \newcommand{\ena}{\end{eqnarray}}
  
\newcommand{\J}{{\cal J}}

\newcommand{\beas}{\begin{eqnarray*}}
\newcommand{\enas}{\end{eqnarray*}}

\newcommand{\p}{\tilde{p}}

\ignore{
\documentclass[11pt,leqno]{article}
\textwidth=17.cm
\textheight=21.6cm
\topmargin=-0.5cm
\oddsidemargin=-0.5cm

}

\newcommand{\be}{\begin{equation}}               
\newcommand{\ee}{\end{equation}}                 
\newcommand{\bi}{\begin{itemize}}
\newcommand{\ei}{\end{itemize}}

\ignore{

}

\newcommand{\beqn}{\begin{eqnarray}}             
\newcommand{\eeqn}{\end{eqnarray}}               
\newcommand{\beq}{\begin{eqnarray*}}             
\newcommand{\eeq}{\end{eqnarray*}}               

\newcommand{\nn}{\nonumber}

\newcommand{\lbl}{\label}
\newcommand{\eq}[1]{$(\ref{#1})$}

\newcommand{\ssb}{\scriptstyle \footnotesize 
                 \begin{array}{c}}

\newcommand{\esb}{\end{array}}

\begin{document}
\renewcommand{\baselinestretch}{1.2}
\markright{
}
\markboth{\hfill{\footnotesize\rm LOUIS H.Y. CHEN, XIAO FANG AND QI-MAN SHAO
}\hfill}
{\hfill {\footnotesize\rm MODERATE DEVIATIONS IN POISSON APPROXIMATION} \hfill}
\renewcommand{\thefootnote}{}
$\ $\par
\fontsize{10.95}{14pt plus.8pt minus .6pt}\selectfont
\vspace{0.8pc}
\centerline{\large\bf MODERATE DEVIATIONS IN POISSON APPROXIMATION:}
\centerline{\large\bf A FIRST ATTEMPT}
\vspace{.4cm}
\centerline{Louis H. Y. Chen, Xiao Fang and Qi-Man Shao}
\vspace{.4cm}
\centerline{\it National University of Singapore,}
\centerline{\it National University of Singapore and Stanford University,}
\centerline{\it and Chinese University of Hong Kong}
\vspace{.55cm}
\fontsize{9}{11.5pt plus.8pt minus .6pt}\selectfont

\begin{quotation}
\noindent {\it Abstract:}
Poisson approximation using Stein's method has been extensively studied in the literature. The main focus has been on bounding the total variation distance.  This paper is a first attempt on moderate deviations in Poisson approximation for right-tail probabilities of sums of dependent indicators. We obtain results under certain general conditions for local dependence as well as for size-bias coupling. These results are then applied to independent indicators, 2-runs, and the matching problem.

\par

\vspace{9pt}
\par
\end{quotation}\par

\newcommand{\Var}{\mbox{Var}}
\renewcommand{\theequation}{\arabic{chapter}.\arabic{equation}}
\newtheorem{definition}{{\bf Definition}}[chapter]
\newtheorem{theorem}{{\bf Theorem}}[chapter]
\newtheorem{prop}{ {\bf Proposition}}[chapter]
\newtheorem{lemma}{{\bf Lemma}}[chapter]
\newtheorem{coro}{{\bf Corollary}}[chapter]
\newtheorem{remark}{{\bf Remark}}[chapter]
\newtheorem{exam}{{\bf Example}}[chapter]

\newcommand{\Proof}{ \noindent {\bf Proof.}\ }

\fontsize{10.95}{14pt plus.8pt minus .6pt}\selectfont
\setcounter{chapter}{1}
\setcounter{equation}{0} 
\noindent {\bf 1. Introduction}

Poisson approximation using Stein's method has been applied to many areas, ranging from computer science to computational biology. The main focus has been on bounding the total variation distance between the distribution of a sum of dependent indicators and the Poisson distribution with the same mean.

Broadly speaking, there are two main approaches to Poisson approximation, the local approach and the size-bias coupling approach.  The local approach was first studied by Chen (1975) and developed further by Arratia, Goldstein and Gordon (1989, 1990), who presented Chen's results in a form which is easy to use, and applied them to a wide range of problems including problems in extreme values, random graphs and molecular biology.  The size-bias coupling approach dates back to Barbour (1982) in his work on Poisson approximation for random graphs. Barbour, Holst and Janson (1992) presented a systematic development of monotone couplings and applied their results to random graphs and many combinatorial problems. A recent review of Poisson approximation by Chatterjee, Diaconis and Meckes (2005) used Stein's method of exchangeable pairs to study classical problems in combinatorial probability. They also reviewed a size-bias coupling of Stein (1986, p. 93).

Although there is a vast literature on Poisson approximation, relatively little has been done on such refinements as moderate deviations.  For sums of independent indicators, moderate deviations have been studied by Barbour, Holst and Janson (1992), Chen and Choi (1992), and Barbour, Chen and Choi (1995).  The latter two actually considered the more general problem of unbounded function approximation and deduced moderate deviations as a special case.  However no such results seem to have been obtained for dependent indicators, probably due to the fact that unbounded function approximation becomes much harder for dependent indicators. Although moderate deviations is a special case of unbounded function approximation, it is of a similar nature as the latter and, as such, it is also a difficult problem for dependent indicators.

This paper is a first attempt on moderate deviations in Poisson approximation for dependent indicators.  We take both the local and the size-bias coupling approach.  Under the local approach we consider locally dependent indicators.  Under the size-bias coupling approach we consider size-bias coupling, which generalizes the monotone couplings of Barbour, Holst and Janson (1992) and the size-bias coupling of Stein (1986).  In both approaches, we consider moderate deviations for right-tail probabilities under certain general conditions.

This paper is organized as follows.  Section 2 contains the main theorems. In Section 3, we apply our main theorems to Poisson-binomial trials, $2$-runs in a sequence of i.i.d. Bernoulli random variables, and the matching problem. As far as we know, the results for the last two applications are new. In Section 4 we prove the main theorems.

\par

\setcounter{chapter}{2}
\setcounter{equation}{0} 

\bigskip

\noindent {\bf 2. Main Theorems}

In this section, we state two general theorems on moderate deviations in Poisson approximation, one under local dependence and the other under size-bias coupling. Let $|\cdot|$ denote the Euclidean norm or cardinality.

\noindent {\bf 2.1 Local dependence}

Local dependence is a widely used dependence structure for Poisson approximation.
We refer to Arratia, Goldstein and Gordon (1989, 1990) for results on the total variation distance and applications.
Here we prove a moderate deviation result.
Let $X_i, i \in \J
$, be random indicators indexed by
$\J$. Let  $W= \sum_{i \in \J} X_i$,
\begin{equation}
 p_i = P(X_i =1), \ \ \mbox{and} \ \   \lambda = \sum_{i \in \J} p_i>0 \ . \label{0}
\end{equation}
Suppose for each $i\in \J$, there exists a subset $B_i$ of $\J$ such that $X_i$ is independent of $\{X_j: j\notin B_i\}$. The subset $B_i$ is called a dependence neighborhood of $X_i$.
Assume that
\begin{equation}
\max_{i \in \J} |B_i| \leq m,\quad  \max_{j\in \J} |\{i: j\in B_i\}|\leq m
\lbl{m},
\end{equation}
and, for some $\delta, \theta>0$,
\begin{equation}
E (\sum_{i\in \J} \sum_{j\in B_i\backslash \{i\}} X_i X_j |W=w)\leq \delta w^2\  \text{for}\  w\leq \theta.
\lbl{delta}
\end{equation}
Let $\p = \max_{i \in \J} p_i$.

\begin{theorem}\lbl{LDt}
Let $W=\sum_{i\in \J} X_i$ be a sum of locally dependent random indicators with dependence neighborhoods $B_i$ satisfying \eq{m} and \eq{delta}. Then there exist absolute positive constants $c, C$ such that for $k\geq \lambda$ satisfying
 \beq
k\leq \theta/Cm,\quad \tilde{p} (1+\xi^2) +\delta \lambda (1+\xi^2 +\frac{\xi^3}{\sqrt{\lambda}})
\leq c/m^2
 \eeq
where $\xi=(k-\lambda)/\sqrt{\lambda}$, we have
 \begin{equation}
 \begin{split}
 &\Bigl| \frac{P(W\geq k)}{P(Y\geq k)}-1 \Bigr| \\
 &\leq Cm^2\Big\{  \tilde{p} (1+\xi^2) +\delta \lambda (1+\xi^2 +\frac{\xi^3}{\sqrt{\lambda}}) \Big\} + C (1\wedge \frac{1}{\lambda})m^2 \exp(- \frac{c \theta}{m})
 \end{split}
 \end{equation}
where $Y\sim Poi(\lambda)$.
\end{theorem}

\begin{remark}
{\rm The main difficulty in applying Theorem \ref{LDt} is to verify the condition \eq{delta}. Intuitively, if for many $i\in \J, j\in B_i\backslash \{i\}$, $p_{ji}:=P(X_j=1|X_i=1)$ is large, then given $W=w$, the $w$ $1$'s tend to appear in clusters, which makes the left-hand side of \eq{delta} large (bounded by $w^2$ in the extreme case). If $p_{ji}$ is small, then the $w$ $1$'s tend to be distributed widely, making the left-hand side of \eq{delta} small ($0$ in the extreme case). It is a challenge to replace the $\delta$ in \eq{delta} by a quantity involving only $\{p_i, p_{ji}: i\in \J, j\in B_i\backslash \{i\}\}$.}
\end{remark}

\bigskip

\noindent {\bf 2.2 Size-bias coupling}

Baldi, Rinott and Stein (1989) and Goldstein and Rinott (1996) used size-bias coupling to prove normal approximation results by Stein's method. In the context of Stein's method for Poisson approximation, size-bias coupling was used implicitly by Stein (1986, page 93), Barbour (1982), Barbour, Holst and Janson (1992, page 23) and Chatterjee, Diaconis and Meckes (2005, page 93). The following definition of size-bias distribution can be found in Goldstein and Rinott (1996).
\begin{definition}
For $W$ a non-negative random variable, $W^s$ has a $W$-size biased distribution if
\begin{equation}
E W f(W) =\lambda E f(W^s)
\lbl{2}
\end{equation}
for all functions $f$ such that the expectations exist.
\end{definition}
We take $W$ to be a non-negative integer-valued random variable, in particular, a sum of random indicators. If we can couple $W$ with $W^s$ on the same probability space, then we have a bound on the total variation distance between $\mathcal{L}(W)$ and a Poisson distribution.
\begin{theorem}\label{t1}
Let $W$ be a non-negative integer-valued random variable with $E W=\lambda >0$. If $W^s$ is defined on the same probability space as $W$ with a $W$-size biased distribution, then 
\begin{equation}
\| {\cal L}(W) - Poi(\lambda)\|_{TV}
\leq (1- e^{-\lambda}) E|W+1 - W^s|.
\lbl{3}
\end{equation}
\end{theorem}
\begin{proof}
Let $h(w)=I(w\in A)$ for $w\in \mathbb{Z}_+$, where $A$ is any given subset of $\mathbb{Z}_+$. Let $f_h$ be the bounded solution (unique except at $w=0$) to the Stein equation
\begin{equation}
 \lambda f(w+1) - w f(w) = h(w) - E h(Y) \lbl{cs}
\end{equation}
where $Y\sim Poi(\lambda)$. It is known that (see, for example,  Barbour, Holst and Janson (1992, page 7))
\begin{equation}\label{t1-1}
\Delta f_h:=\sup_{j\in \mathbb{Z}_+, j\geq 1} |f_h(j+1)-f_h(j)|\leq \lambda^{-1}(1-e^{-\lambda}).
\end{equation}
From \eq{cs} and the fact that $W^s$ is coupled with $W$ and has the $W$-size biased distribution, we have
\begin{equation*}
\begin{split}
|P(W\in A)-P(Y\in A)|& = |\lambda E f_h(W+1) - E W f_h(W)|\\
& = \lambda |E (f_h(W+1)-f_h(W^s))|\\
&\leq \lambda \Delta f_h E |W+1-W^s|\\
&\leq (1-e^{-\lambda})E|W+1-W^s|,
\end{split}
\end{equation*}
where the first inequality is obtained by writing $f_h(W+1)-f_h(W^s)$ as a telescoping sum and using the definition of $\Delta f_h$, along with the fact that $W^s\geq 1$. The second inequality follows from \eq{t1-1}.
Taking supremum over $A$ yields \eq{3}.
\end{proof}
Similar results as Theorem \ref{t1} can be found in Barbour, Holst and Janson (1992) and Chatterjee, Diaconis and Meckes (2005).
In order for the bound \eq{3} to be useful, we need to couple $W$ with $W^s$ such that $E|W+1-W^s|$ is small. A general way of constructing such size-bias couplings for sums of random indicators is as follows; see, for example, Goldstein and Rinott (1996). Let $\mathbb{X}=\{X_i\}_{i \in \J
}$ be $\{0, 1 \}$-valued random variables with $P(X_i=1)=p_i$, $\lambda=\sum_{i\in \J} p_i$, and let $W= \sum_{i \in \J} X_i$. Let $I$ be independent of $\mathbb{X}$ with $P(I=i)=p_i/\lambda$. Given $i\in \J$, construct $\mathbb{X}^i=\{X_j^i\}_{j \in \J}$ on the same probability space as $\mathbb{X}$ such that
\begin{equation*}
\mathcal{L}(X_j^i: j\in \J)=\mathcal{L}(X_j: j\in \J | X_i=1).
\end{equation*}
Then $W^s=\sum_{j\in \J} X_j^I$ has the $W$-size biased distribution.

\begin{theorem}\label{t2}
Let $W$ be a non-negative integer-valued random variable with $E W =\lambda>0$.
Let $W^s$ be defined on the same probability space as $W$ with a $W$-size biased distribution.
Assume that
$\Delta:=W+1-W^s \in \{-1, 0, 1\}$ and that there are non-negative constants $\delta_1, \delta_2$ such that
\begin{equation}
P(\Delta = - 1 \ | \ W) \leq \delta_1, \quad P(\Delta =  1 \ | \ W) \leq \delta_2 W.
\lbl{c-1}
\end{equation}
For integers $k\geq \lambda$, let $\xi=(k-\lambda)/\sqrt{\lambda}$.
Then there exist absolute positive constants $c, C$, such that for
$(\delta_1+\delta_2 \lambda)(1+\xi^2)\leq c$, we have
 \begin{equation}\label{t2-1}
\Bigl| \frac{P(W\geq k)}{P(Y\geq k)}-1 \Bigr| \leq C (\delta_1+\delta_2 \lambda)(1+\xi^2),
 \end{equation}
where $Y \sim Poi(\lambda)$.
\end{theorem}
The conditions of Theorem \ref{t2} do not hold for all size-bias couplings. Nevertheless, in Section 3, we are able to apply Theorem \ref{t2} to prove moderate deviation results for Poisson-binomial trials and the matching problem. It is possible to replace the upper bounds in \eq{c-1} by any polynomial function of $W$, resulting in a change of the upper bound in \eq{t2-1}. However we will not pursue this in this paper.

\par

\setcounter{chapter}{3}
\setcounter{equation}{0} 

\bigskip

\noindent {\bf 3. Applications}

In this section, we apply our main results to Poisson-binomial trials, 2-runs in a sequence of i.i.d. indicators and the matching problem.

\bigskip

\noindent{\bf 3.2. Poisson-binomial trials}

Let $X_i, i\in \J$, be independent with $P(X_i=1)=p_i=1-P(X_i=0)$. Set $\lambda=\sum_{i\in \J}p_i$ and $\tilde{p}=\sup_{i\in \J} p_i$. Let $W=\sum_{i\in \J} X_i$. Following the construction in Section 2.1, $W^s$ in \eq{2} can be constructed as $W^s=W-X_I+1$, where $I$ is independent of $\{X_i: i\in \J\}$ and $P(I=i)=p_i/\lambda$ for each $i\in \J$. Therefore, $\Delta=W+1-W^s=X_{I}$ and condition \eq{c-1} is satisfied with $\delta_1=0, \delta_2=\tilde{p}/\lambda$. Applying Theorem \ref{t2}, there exist absolute positive constants $c, C$ such that
 \begin{equation}\label{Pbt}
 \Bigl| \frac{P(W\geq k)}{P(Y\geq k)}-1 \Bigr| \leq C \p (1+\xi^2)
 \end{equation}
for integers $k\geq \lambda$ and $\tilde{p}(1+\xi^2) \leq c$ where $Y\sim Poi(\lambda)$ and $\xi=(k-\lambda)/\sqrt{\lambda}$. The range $\tilde{p}(1+\xi^2) \leq c$ is optimal for the i.i.d. case where $p_i=\tilde{p}$ for all $i\in \J$ (see Theorem 9.D of Barbour, Holst and Janson (1992, page 188) and Corollary 4.3 of Barbour, Chen and Choi (1995)).
\begin{remark}
{\rm The moderate deviation result \eq{Pbt} also follows from Theorem \ref{LDt} for sums of locally dependent random variables.}
\end{remark}

\bigskip

\noindent{\bf 3.1. $2$-runs.}

Let $\{\xi_1,\dots, \xi_n\}$ be i.i.d. $Bernoulli(p)$ variables with $n>10$, $p<1/2$. For each $i\in \{1,\dots, n\}$, let $X_i=\xi_i \xi_{i+1}$ where $\xi_{j+n}=\xi_{j-n}=\xi_j$ for any integer $j\in \{1,\dots, n\}$. Take $W=\sum_{i=1}^n X_i$ with mean $\lambda=np^2$.
Then $W$ is a sum of locally dependent random variables with $m=3$ where $m$ is defined in \eq{m}. For each $i\in \{1,\dots, n\}$ and any positive integer $w\leq cnp$ for some sufficiently small constant $c<1/50$ to be chosen later, we write
 \beq
 &&P(X_i=1, X_{i+1}=1, W=w)\\
 &=&\sum_{m_1\geq 0, m_2\geq 1 \atop m_1+m_2<w} P(X_{i-m_1}=\dots=X_{i+m_2}=1, X_{i-m_1-1}=X_{i+m_2+1}=0, W=w)\\
 &=:& \sum_{m_1\geq 0, m_2\geq 1 \atop m_1+m_2<w} a_{m_1, m_2}
 \eeq
where the sum is over integers.
By writing
 \beq
 a_{m_1, m_2}= p^{m_1+m_2+2}(1-p)^2 P(\sum_{i=1}^{n-(m_1+m_2+5)} X_i=w-(m_1+m_2+1)),
 \eeq
we have for $m_1+m_2+1<w$,
 \begin{equation}\label{2runs-1}
 \frac{a_{m_1, m_2+1}}{a_{m_1,m_2}}
 =p\frac{P(\sum_{i=1}^{n-(m_1+m_2+6)}X_i=w-(m_1+m_2+2) )}{P(\sum_{i=1}^{n-(m_1+m_2+5)}X_i=w-(m_1+m_2+1) )}\leq Cp\frac{w}{\lambda}
 \end{equation}
for some positive constant $C$. The last inequality is proved by observing that for each event
\beq
&\{X_i=x_i: 1\leq i\leq n-(m_1+m_2+6)\}\\
& \text{with}\quad  \sum_{i=1}^{n-(m_1+m_2+6)} x_i=w-(m_1+m_2+2),
\eeq
we can change one of the $\dots 000 \dots$ to $\dots 010 \dots$ and let $x_{n-(m_1+m_2+5)}=0$, thus resulting in an event
\beq
&\{X_i=x_i: 1\leq i\leq n-(m_1+m_2+5)\}\\
& \text{with}\quad  \sum_{i=1}^{n-(m_1+m_2+5)} x_i=w-(m_1+m_2+1),
\eeq
the probability of which is at least $c_1 p^2$ times the probability of the original event for an absolute positive constant $c_1$. Summing over the probabilities of all the events obtained in this way, and correcting for the multiple counts, yields the inequality in \eq{2runs-1}.
By choosing $c$ to be small,
 \beq
 \frac{a_{m_1, m_2+1}}{a_{m_1,m_2}}\leq \frac{1}{4}.
 \eeq
Similarly,
 \beq
 \frac{a_{m_1+1, m_2}}{a_{m_1,m_2}}\leq \frac{1}{4}.
 \eeq
Therefore,
 \beq
 P(X_i=1, X_{i+1}=1, W=w)\leq C a_{0,1} \leq Cp^3 P(\sum_{i=1}^{n-6} X_i=w-2).
 \eeq
Similar to \eq{2runs-1},
 \beq
 P(\sum_{i=1}^{n-6} X_i=w-2) \leq C (w^2/\lambda^2) P(W=w).
 \eeq
Therefore,
 \beq
 && \sum_{i=1}^n \sum_{j=i-1, i+1}E(X_i X_j|W=w)\\
 &=& 2n P(X_i=X_{i+1}=1,W=w)/P(W=w) \\
 &\leq& C n p^3 w^2/\lambda^2 =\frac{C}{np}w^2
 \eeq
for $w\leq cnp$ with sufficiently small $c$. Applying Theorem \ref{LDt}, there exist absolute positive constants $c, C$, such that for $k\geq \lambda$ and $p+p \xi^2+\xi^3/\sqrt{n} \leq c$, where $\xi=(k-\lambda)/\sqrt{\lambda}$,
 \beqn
 \Bigl| \frac{P(W\geq k)}{P(Y\geq k)}-1 \Bigr|\leq C(p+p \xi^2+\xi^3/\sqrt{n})\lbl{star},
 \eeqn
where $Y\sim Poi(\lambda)$. We remark that if $\lambda \asymp O(1)$, then the range of $\xi$ is of order $O(n^{1/6})$.
\begin{remark}
{\rm Although the rate $O(n^{1/6})$ may not be optimal, we have not seen a result like \eq{star} in the literature. Our argument for $2$-runs can be extended to study $k$-runs for $k\geq 3$.}
\end{remark}

\bigskip

\noindent{\bf 3.3. Matching problem}

For a positive integer $n$, let $\pi$ be a uniform random permutation of $\{1,\dots,n\}$. Let $W=\sum_{i=1}^n \delta_{i\pi(i)}$ be the number of fixed points in $\pi$. In Chatterjee, Diaconis and Meckes (2005), $W^s$ satisfying \eq{2} was constructed as follows. First pick $I$ uniformly from $\{1,\dots, n\}$, and then set
 \begin{equation*}
 \pi^s(j)=
 \begin{cases}
  I & \text{if}\  j=I \\
  \pi(I) & \text{if}\ j=\pi^{-1}(I) \\
  \pi(j) & \text{otherwise}.
 \end{cases}
 \end{equation*}
Take $W^s=\sum_{i=1}^n \delta_{i\pi^s(i)}$. With $\Delta=W+1-W^s$, we have
 \beq
 P(\Delta=1|W)=W/n,\quad P(\Delta=-1|W)=E(2a_2|W)/n\leq 2/n,
 \eeq
where $a_2$ is the number of transpositions of $\pi$, and the last inequality follows since
 \beq
 E(2a_2|W)=(n-W)/(n-W-1)\leq 2
 \eeq
for $n-W\geq 2$, and $E(2a_2|W)=0$ for $n-W\leq 1$. By Theorem \ref{t2} with $\lambda=1$, there exist absolute positive constants $c, C$ such that for all positive integers $k$ satisfying $k^2/n\leq c$,
 \beq
 \Bigl| \frac{P(W\geq k)}{P(Y\geq k)}-1 \Bigr| \leq Ck^2/n.
 \eeq
We remark that the order $O(1/n)$ is the same as that of the total variation bounds in Barbour, Holst and Janson (1992) and Chatterjee, Diaconis and Meckes (2005). As remarked in those papers, this order is not optimal; it is an open problem to prove the actual order $O(2^n/n!)$ using Stein's method.

\par

\setcounter{chapter}{4}
\setcounter{equation}{0} 

\bigskip

\noindent {\bf 4. Proofs}

We use $c, C$, to denote absolute positive constants whose values may be different at each appearance.

\begin{lemma}\label{lem}
For any integer $w\geq \lambda > 0$,
\begin{equation}\label{lem-1}
\sum_{j=0}^\infty \lambda^j { w! (j+1) \over (j+w+1)!} \leq C.
\end{equation}
\end{lemma}
\begin{proof} We first bound $\lambda^j$ by $w^j$. Next, by expanding the product $(w+j+1)\times \dots \times (w+1)$ in terms of $w$ and then bounding it below by $w^{j+1}$ and $cj^4 w^{j-1}$, respectively, in the expansion, we have
 \beq
 \sum_{j=0}^\infty \lambda^j { w! (j+1) \over (j+w+1)!}
 &\leq& \sum_{j=0}^\infty w^j { j+1 \over (w+j+1)\times \dots \times (w+1)} \\
 &\leq& \sum_{ j \leq \sqrt{w}} \frac{j+1}{w} +\sum_{j>\sqrt{w}} \frac{j+1}{cj^4/w}   \\
 &\leq& C,
 \eeq
 as desired.
\end{proof}

\begin{lemma}\label{lem0}
Let $Y\sim Poi(\lambda)$ with $\lambda>0$. Then we have
\begin{equation}\label{lem0-1}
P(Y\geq k) \geq c >0\  \text{for all integer}\  k< \lambda,
\end{equation}
\begin{equation}\label{lem0-2}
\frac{P(Y\geq k)}{P(Y\geq k-1)} \geq \frac{\lambda}{\lambda+k}\  \text{for all integer}\ k\geq 1,
\end{equation}
\begin{equation}\label{lem0-3}
P(Y\geq k) \leq P(Y=k) \frac{k+1}{k-\lambda+1}\  \text{for all integer}\ k>\lambda-1.
\end{equation}
\end{lemma}

\begin{proof}
The inequality in \eq{lem0-1} is trivial when $\lambda<1$ or $1\leq \lambda \le C$ for some absolute constant $C$. When $\lambda>C$, we can use normal approximation to prove \eq{lem0-1}.

For \eq{lem0-2}, noting that
 \beq
 P(Y\geq k)&=&P(Y=k) (1+\frac{\lambda}{k+1}+\frac{\lambda^2}{(k+1)(k+2)}+\cdots)\\
 &\geq& \frac{\lambda+k+1}{k+1} P(Y=k),
 \eeq
we have
\beq
\frac{P(Y\geq k)}{P(Y\geq k-1)}=1-\frac{P(Y=k-1)}{P(Y\geq k-1)}\geq 1-\frac{k}{\lambda+k}=\frac{\lambda}{\lambda+k}.
\eeq

The inequality in \eq{lem0-3} follows by observing that
 \beq
 P(Y\geq k)&=&P(Y=k) (1+\frac{\lambda}{k+1}+\frac{\lambda^2}{(k+1)(k+2)}+\cdots)\\
 &\leq& P(Y=k) (1+  \frac{\lambda}{k+1}+\frac{\lambda^2}{(k+1)^2} + \cdots  ) \\
 & = & P(Y=k) \frac{k+1}{k-\lambda+1}.
 \eeq
\end{proof}

The bounded solution $f_h$ (unique except at $w=0$) to the Stein equation
\begin{equation}
 \lambda f(w+1) - w f(w) = h(w) - E h(Y), \lbl{cs1}
\end{equation}
where $Y\sim Poi(\lambda)$ and $h(w)=I\{w\geq k\}$ for fixed integer $k\geq \lambda>0$, is
\beq
f_h(w) &  = & - { e^\lambda (w-1)! \over \lambda^w} E( h(Y) - Eh(Y)) I\{Y \geq w\}\nn \\
& = &
\left\{
\begin{array}{ll}
- { e^\lambda (w-1)! \over \lambda^w} ( 1- P(Y \geq k))  P( Y \geq w), \ \  w \geq k, \\
 \vspace{-.3cm} \\
 - { e^\lambda (w-1)! \over \lambda^w} P(Y \geq k)  P( Y \leq w-1), \ \  0< w \leq k.
  \end{array} \right.
  \eeq
Although $f_h(0)$ does not enter into consideration, we set $f_h(0):=f_h(1)$.

For $w \geq k$,
\beq
{ f_h(w) - f_h(w+1) \over 1- P(Y \geq k)}   & =&
 { e^\lambda w! \over \lambda^{w+1}} P( Y \geq w+1) -
  { e^\lambda (w-1)! \over \lambda^w} P( Y \geq w)\nn \\
  & =& \sum_{j=w+1}^\infty { w ! \over j!} \lambda^{j-w-1}
  - \sum_{j=w}^\infty { (w-1) ! \over j!} \lambda^{j-w}  \nn \\
  & =& \sum_{j=0}^\infty \lambda^j ({ w! \over (j+w+1)!} - {(w-1)! \over (j+w)!}
 )\nn \\
 & =& - \sum_{j=0}^\infty \lambda^j { (w-1)! (j+1) \over (j+w+1)!},
 \eeq
 and hence by \eq{lem-1},
 \begin{equation}
 0 < f_h(w+1) - f_h(w) \leq { C \over w} \ \ \mbox{for} \ w \geq k .\lbl{up-1}
 \end{equation}

 For $0\leq w \leq k-1$,
 \begin{equation*}
 { f_h(w) - f_h(w+1) \over P(Y\geq k)}
  =g_1(w).
\end{equation*}
where
\begin{equation}
g_1(w)={ e^\lambda w! \over \lambda^{w+1}}  P( Y \leq w)
- { e^\lambda (w-1)! \over \lambda^{w}}  P( Y \leq w-1)\lbl{up-3}
\end{equation}
and $g_1(0):=0$.

Let $W$ be a non-negative integer-valued random variable with $E W=\lambda>0$,  and let $Y\sim Poi(\lambda)$. Define
 \begin{equation}\lbl{LDt-1}
 \eta_k:=\sup_{\lambda \leq r \leq k} \frac{P(W\geq r)}{P(Y\geq r)}.
 \end{equation}
By \eq{lem0-1},
 \begin{equation}\label{t2-1b}
 \sup_{0 \leq r \leq k} \frac{P(W\geq r)}{P(Y\geq r)} \leq \eta_k + C.
 \end{equation}

\begin{lemma}\label{R1}
The function $g_1$ is non-negative, non-decreasing and
\begin{equation}\label{r1}
g_1(w)\leq \frac{1}{\lambda}+\frac{(w-1)!(w-\lambda)_+}{\lambda^{w+1}}e^\lambda
\end{equation}
for all $w\geq 1$ where $x_+$ denotes the positive part of $x$.
\end{lemma}
\begin{proof}
For $w\geq 1$, $g_1(w)$ can be expressed as
\begin{equation*}
 \begin{split}
&{ e^\lambda w! \over \lambda^{w+1}}  P( Y \leq w)
- { e^\lambda (w-1)! \over \lambda^{w}}  P( Y \leq w-1)  \\
 = &{ e^\lambda \over \lambda^{w+1}} \int_\lambda^\infty  x^{w} e^{-x} dx
-  { e^\lambda \over \lambda^{w}} \int_\lambda^\infty  x^{w-1} e^{-x} dx  \\
 = &e^\lambda \int_1^\infty x^{w-1} ( x-1)  e^{-\lambda x} dx  \\
 = &\int_0^\infty x (1+x)^{w-1}  e^{- \lambda x} dx,
\end{split}
\end{equation*}
from which $g_1$ is non-negative and non-decreasing. Also for $w\geq 1$,
\begin{equation*}
 \begin{split}
&{ e^\lambda w! \over \lambda^{w+1}}  P( Y \leq w)
- { e^\lambda (w-1)! \over \lambda^{w}}  P( Y \leq w-1)  \\
 =& \frac{e^\lambda w!}{\lambda^{w+1}} P(Y=w)+\big( \frac{e^\lambda w!}{\lambda^{w+1}}-{ e^\lambda (w-1)! \over \lambda^{w}} \big) P(Y\leq w-1) \\
 \leq &\frac{1}{\lambda}+\frac{(w-1)!(w-\lambda)_+}{\lambda^{w+1}}e^\lambda.
\end{split}
\end{equation*}
\end{proof}
\begin{lemma}\label{R2}
For any non-negative and non-decreasing function $g: \{0,1,2,\dots\}\rightarrow \mathbb{R}$ and any $k\geq 0$, we have
\begin{equation}\label{r2}
Eg(W\wedge k)\leq C(\eta_k+1)Eg(Y\wedge k).
\end{equation}
\end{lemma}
\begin{proof}
Write
\beq
g(W\wedge k) =g(0)+\sum_{j=1}^k (g(j)-g(j-1))I(W\geq j).
\eeq
From \eq{t2-1b} and the fact that $g$ is non-decreasing, we have
\beq
Eg(W\wedge k) &\leq & g(0)+C(\eta_k+1)\sum_{j=1}^k (g(j)-g(j-1))P(Y\geq j)\\
&=&C(\eta_k+1) Eg(Y\wedge k).
\eeq
\end{proof}
\begin{lemma}\label{R3}
For all $k\geq 0$, we have
\begin{equation}\label{r3-1}
E g_1((W+1)\wedge k)\leq C(\eta_k+1)\big( \frac{1}{\lambda}+\frac{(k+1-\lambda)_+^2}{\lambda^2} \big),
\end{equation}
\begin{equation}\label{r3-2}
E [ (W\wedge k)g_1(W\wedge k)]\leq C(\eta_k+1)\big( 1+\frac{(k-\lambda)_+^2}{\lambda} \big),
\end{equation}
\begin{equation}\label{r3-3}
E[(W\wedge k)^2 g_1(W\wedge k)]\leq C(\eta_k+1)\big( \lambda+(k-\lambda)_+^2+\frac{(k-\lambda)_+^3}{\lambda} \big).
\end{equation}
\end{lemma}
\begin{proof}
The case $k=0$ is trivial. Let $k\geq 1$. For any $p\in \{0,1\}, q\geq 0$, by \eq{r2} and \eq{r1},
\beq
&&E\big[((W+p)\wedge k)^q g_1((W+p)\wedge k)  \big]\\
&\leq & C(\eta_k+1)
E\big[ ((Y+p)\wedge k)^q g_1((Y+p)\wedge k)   \big]\\
&\leq& C(\eta_k+1) \big( \frac{k^q}{\lambda}+A(k,p,q) +B(k,q) \big)
\eeq
where
\beq
A(k,p,q)=E\big[ \frac{(Y+p)^q(Y+p-1)!(Y+p-\lambda)_+}{\lambda^{Y+p+1}}e^\lambda I(1-p\leq Y\leq k-1) \big],
\eeq
\beq
B(k,q)=\frac{k^q (k-1)! (k-\lambda)_+}{\lambda^{k+1}}e^\lambda P(Y\geq k).
\eeq
Using \eq{lem0-3}, $B(k,q)$ is bounded by
\beq
B(k,q)\leq \frac{k^q}{\lambda}\frac{(k-\lambda)_+}{k}\frac{k+1}{k-\lambda+1}\leq \frac{k^q}{\lambda}.
\eeq
The relevant special cases of the quantities $A(k,p,q)$ are
\beq
A(k,1,0)=\sum_{w=0}^{k-1}\frac{(w+1-\lambda)_+}{\lambda^2}\leq \frac{(k+1-\lambda)_+^2}{2\lambda^2},
\eeq
\beq
A(k,0,1)=\sum_{w=1}^{k-1}\frac{(w-\lambda)_+}{\lambda}\leq \frac{(k-\lambda)_+^2}{2\lambda},
\eeq
\beq
A(k,0,2)&=&\sum_{w=1}^{k-1} \frac{w(w-\lambda)_+}{\lambda}=\sum_{w=1}^{k-1}\big[(w-\lambda)_+ +\frac{(w-\lambda)_+^2}{\lambda}  \big]\\
&\leq& \frac{(k-\lambda)_+^2}{2}+\frac{(k-\lambda)_+^3}{3\lambda}.
\eeq
Combining these bounds and observing that $(k-\lambda)_+\leq C(\lambda+(k-\lambda)_+^2)$
yields the desired result.
\end{proof}

We first prove of Theorem \ref{t2}, which is easier than Theorem \ref{LDt}.

\begin{proof}[Proof of Theorem \ref{t2}]

For fixed integer $k \geq \lambda $, let $h(w)= I\{w \geq k \}$.
Observe that by \eq{2}, for general $f$,
\begin{equation}
E(  \lambda f(W+1) - Wf(W))
 =  \lambda E(f(W+1) - f(W^s) ). \lbl{t2-2}
\end{equation}
In particular, for $f:=f_h$,
\begin{equation}\lbl{t2-3}
\begin{split}
 Eh(W) - Eh(Y)   = &   \lambda E(f(W+1) - f(W^s)) \\
  :=& H_1 + H_2
 \end{split}
 \end{equation}
 where
 \beq
 H_1 & =& \lambda E\big[ (f(W+1) - f(W+2))I\{ \Delta =-1\} \big], \\
 H_2 & =& \lambda  E \big[ (f(W+1) - f(W))I\{ \Delta =1\} \big] .
\eeq
Using \eq{c-1}, the definition of $\eta_k$ in \eq{LDt-1}, and the properties of $f_h$, $H_1$ is bounded by
 \beq
 |H_1| &\leq & \lambda \delta_1 E \big[ |f(W+1)-f(W+2)| (I(W+1\geq k)+ I(W+1\leq k-1)) \big]\nn \\
   &\leq& \lambda \delta_1 \frac{C P(W\geq k-1)}{k} +\lambda \delta_1 P(Y\geq k) E \big[ I(W+1\leq k-1) g_1(W+1)\big]\nn \\
   &\leq& \lambda \delta_1 \frac{C P(W\geq k-1)}{k} +\lambda \delta_1 P(Y\geq k) E  g_1((W+1)\wedge (k-1))\nn \\
    &\leq& C P(Y\geq k) \delta_1 (\eta_k+1) + CP(Y\geq k) \delta_1 (1+\frac{(k-\lambda)^2}{\lambda}) (\eta_k+1)
 \eeq
where we used \eq{t2-1b}, \eq{lem0-2} and \eq{r3-1}.

Similarly,
 \beq
 |H_2| &\leq & \lambda \delta_2 E \big[ W|f(W)-f(W+1)| (I(W\geq k) + I(W\leq k-1))\big] \\
           &\leq & C\lambda \delta_2  P(W\geq k) + \lambda \delta_2 P(Y\geq k) E \big[ I(W\leq k-1) Wg_1(W) \big] \\
           &\leq & C\lambda \delta_2  P(W\geq k) + \lambda \delta_2 P(Y\geq k) E \big[ (W\wedge (k-1))g_1(W\wedge (k-1)) \big] \\
           &\leq&  C P(Y\geq k)\lambda \delta_2 \eta_k + C P(Y\geq k) \delta_2 (\lambda+(k-\lambda)^2)(\eta_k+1).
 \eeq
by \eq{LDt-1} and \eq{r3-2}. Therefore, 
\beq
 |\frac{P(W\geq k)}{P(Y\geq k)}-1|&\leq& C(\eta_k+1) (\delta_1+ \delta_2 \lambda)(1 + \xi^2).
\eeq
Since the right-hand side here is increasing in $k$, we have 
\beq
 \eta_k-1 &\leq& C(\eta_k+1) (\delta_1+ \delta_2 \lambda)(1 + \xi^2).
\eeq
The bound in \eq{t2-1} is proved by solving this recursive inequality.
\end{proof}

\begin{proof}[Proof of Theorem \ref{LDt}]
From \eq{cs1} and the definition of the neighborhood $B_i$, we have
 \beq
    &&P(W\geq k) - P(Y\geq k)\\
  & = & \sum_{i\in \J} E X_i[ f(V_i+1)-f(W)]+\sum_{i\in \J} p_i E [f(W+1)-f(V_i+1)] \\
    & =: & H_3+ H_4,
 \eeq
where $V_i:=\sum_{j\notin B_i} X_j$.

We bound $H_4$ first. Write
$ \{X_k: k\in B_i\}= \{X_{ij}: 1\leq j\leq |B_i|\}$, where $|B_i|$ is the cardinality of $B_i$ and $X_{i, |B_i|}:=X_i$. Let
\beq
V_{ij}:=V_i+\sum_{l=1}^{j-1}X_{il}+1.
\eeq
From the definition, if $X_{ij}=1$, then $W\geq V_{ij}$.
By the definitions of $\p, m$ and the properties of $f$,
  \beq
 |H_4| &\leq & \sum_{i\in \J}
  p_i E\big\{ \sum_{j=1}^{|B_i|}X_{ij} \bigl| f(V_{ij}) - f(V_{ij}+1)    \bigr|  \big[ I(V_{ij}\geq k) +I(V_{ij}\leq k-1) \big]\big\} \\
   &\leq& \tilde{p} E\Big\{ \sum_{i\in \J} \sum_{j=1}^{|B_i|} X_{ij} \big[ \frac{C I(V_{ij} \geq k)}{V_{ij}}  + P(Y\geq k) g_1(V_{ij}) I(V_{ij} \leq k-1)   \big] \Big\}\\
   &\leq& \tilde{p} E\Big\{\sum_{i\in \J} \sum_{j=1}^{|B_i|} X_{ij} \big[ \frac{Cm I(W\geq k)}{W}\\
   &&\kern8em+ P(Y\geq k)g_1(W\wedge (k-1)) I(W\leq k+m) \big]  \Big\} \\
   &\leq& m\tilde{p} \Big\{ Cm P(W\geq k)\\
   &&\qquad+ P(Y\geq k) E\big[W I(k\leq W\leq k+m)g_1(k-1) \big] \\
   &&\qquad+ P(Y\geq k) E\big[ (W\wedge (k-1)) g_1(W\wedge(k-1)) \big] \Big\}.
 \eeq
By \eq{LDt-1}, \eq{r1}, \eq{lem0-3} and \eq{r3-2},
 \beq
 |H_4|\leq C P(Y\geq k)m^2 \tilde{p}  (\eta_k+1) \big[ 1 + \frac{(k-\lambda)^2}{\lambda}  \big].
 \eeq
Let $c_1 \geq 1$ be an absolute constant to be chosen later such that $c_1 k m < \theta$. We have
  \beq
  |H_3| &\leq& \sum_{i\in \J} E \Big\{ X_i \sum_{j=1}^{|B_i|-1} X_{ij} \bigl| f(V_{ij}) - f(V_{ij}+1)    \bigr| \\
  &&  \times \big[ I(W\leq c_1 km) +I(c_1 k m<W \leq \theta) +I (W>\theta)   \big] \Big\} \\
  &=:& H_{3,1} + H_{3,2} +H_{3,3}.
  \eeq
By \eq{delta}, $H_{3,1}$ can be bounded similarly as for $|H_4|$ as
 \beq
 H_{3,1}
 &\leq& \sum_{i\in \J} E\Big\{ X_i \sum_{j=1}^{|B_i|-1} X_{ij} \big[ \frac{CI(V_{ij} \geq k)}{V_{ij}} I(W\leq c_1 k m) \\
 &&\kern9em+ P(Y\geq k) g_1(V_{ij}) I(V_{ij} \leq k-1)  \big]  \Big\} \\
 &\leq& \sum_{i\in \J} E \Big\{ X_i \sum_{j=1}^{|B_i|-1} X_{ij} \big[\frac{C m I(W\geq k)}{W} I(W\leq c_1 k m) \\
 &&\kern9em+P(Y\geq k) g_1(W\wedge(k-1)) I(W\leq k+m)  \big]   \Big\} \\
 &\leq& Cm \delta E \big[ W  I(k\leq W \leq c_1 k m)\big] \\
 &&+ \delta P(Y\geq k) E\big[ W^2 I(k\leq W\leq k+m) g_1(k-1)\big] \\
 &&+\delta P(Y\geq k) E \big[ W^2 I(1\leq W \leq k-1) g_1(W)\big]\\
 &\leq& C P(Y\geq k) (\eta_k+1) \delta m^2 (\lambda+(k-\lambda)^2+\frac{(k-\lambda)^3}{\lambda})
 \eeq
 where we used \eq{r3-3} in the last inequality.
 Similarly,
 \beq
 H_{3,2}&\leq& C m\delta E W I(c_1 km<W\leq \theta) \\
 &&+C P(Y\geq k) (\eta_k+1) \delta m^2 (\lambda+(k-\lambda)^2+\frac{(k-\lambda)^3}{\lambda}).
 \eeq
From \eq{l3.1a} of Lemma \ref{l3.1}, proved later, there exists an absolute positive constant $C$ such that for $c_1> C$ and $k< \theta/Cm$,
 \beq
 Cm\delta E W I(c_1 km<W\leq \theta)&\leq& Cm^2\delta E [WI(W>c_1km)]\\
 &\leq & Cm^2\delta P(Y\geq k).
 \eeq
By \eq{l3.1a} and the upper bound $|f(w)-f(w+1)|\leq 1\wedge \frac{1}{\lambda}$ for all integers $w\geq 1$ (see, for example, Barbour, Holst and Janson (1992)),
 \beq
 H_{3,3}
 &\leq& P(Y\geq k) (1\wedge \frac{1}{\lambda})m^2\exp(- \frac{c \theta}{m}).
 \eeq
Therefore,
 \beq
 |H_3|&\leq& CP(Y\geq k)(\eta_k+1) \delta m^2 (\lambda+(k-\lambda)^2 + \frac{(k-\lambda)^3}{\lambda})\\
 &&+P(Y\geq k)(1\wedge \frac{1}{\lambda})m^2 \exp(- \frac{c \theta}{m}).
 \eeq
From the bounds on $|H_3|$ and $|H_4|$, we have
 \beq
 |\frac{P(W\geq k)}{P(Y\geq k)}-1|&\leq& C(\eta_k+1) m^2 \Big\{ \frac{\tilde{p}}{\lambda} (\lambda+(k-\lambda)^2) +\delta (\lambda+ (k-\lambda)^2 + \frac{(k-\lambda)^3}{\lambda})  \Big\} \\
 && +(1\wedge \frac{1}{\lambda}) m^2 \exp(-C \theta).
 \eeq
Since the right-hand side of this bound is increasing in $k$, we have
 \beq
 \eta_k-1&\leq& C(\eta_k+1) m^2 \Big\{ \frac{\tilde{p}}{\lambda} (\lambda+(k-\lambda)^2) +\delta (\lambda+ (k-\lambda)^2 + \frac{(k-\lambda)^3}{\lambda})  \Big\} \\
 && +(1\wedge \frac{1}{\lambda}) m^2 \exp(- \frac{c \theta}{m}).
 \eeq
Solving the above inequality yields Theorem \ref{LDt}.
\end{proof}
For the next lemma, we need a Bennett-Hoeffding inequality.
Let $\{\xi_i, 1 \leq i \leq n\}$ be independent random variables. Assume that
$E\xi_i \leq 0$, $\xi_i \leq a ( a>0)$ for each $ 1 \leq i \leq n$, and
$\sum_{i=1}^n E\xi_i^2 \leq B_n^2$. Then for $ x>0$
$$
P(\sum_{i=1}^n \xi_i \geq x)
\leq \exp( - { B_n^2 \over a^2}
\{ ( 1+ { a x \over B_n^2}) \log ( 1+ { a x \over B_n^2}) - { a x \over B_n^2}\})
$$
In particular, for $x > 4 B_n^2 /a$
\begin{equation}
P(\sum_{i=1}^n \xi_i \geq x)
\leq \exp( - { x \over 2a} \log ( 1+ { a x \over B_n^2}) )
\lbl{Benn-2}
\end{equation}

\begin{lemma} \lbl{l3.1}
Let $W$ be defined as in Theorem \ref{LDt}. Then there exists an absolute constant $C$ such that for $\theta>C k m$, we have
 \begin{equation}\label{l3.1a}
 E W I(W>x)\leq C m \exp(-\frac{x}{8m}\log(1+\frac{x}{2m\lambda})).
 \end{equation}
\end{lemma}
\begin{proof}
We follow the proof of Lemma 8.2 in Shao and Zhou (2012).
Separate $\J$ into $\J_l, 1 \leq l \leq m$, such that for each $l$,
$X_i, i \in \J_l$ are independent. This can be done by coloring $\{X_i: i\in \J\}$ one by one, and in step $j$ we color $X_j$ such that it is independent of those $\{X_i: i<j\}$ with the same color. The total number of colors used can be controlled by $m$ because of \eq{m}. Write
$W_{l} =\sum_{i \in \J_l} X_i$. Then for $y>0$,
\beq
EW I( W > 2 y m )  & =& 2y m  P(W >2 y m   )
+ 2 m  \int_{y }^\infty P(W > 2 t m ) dt \nn \\
& \leq & 2 E (W -y m)^{+}
+ 2  \int_{y}^{\infty}  { 1 \over t} E(W-tm )_+ dt \nn \\
& \leq & 2 \sum_{1 \leq l \leq m} E(W_l - y )_+
+ 2 \sum_{1 \leq l \leq m} \int_{y}^\infty { 1 \over t}  E(W_l - t )_+ dt
\eeq
\ignore{
\beq
P(W > 2t) & \leq & { 1 \over t} E(W-t)_+ \nn \\
& \leq & { 1 \over t} \sum_{1 \leq l \leq m} E( W_l - t_l)_+ \nn
\eeq
}
For $s >   5 \lambda_l:=5 \sum_{i\in \J_l} p_i $, by \eq{Benn-2},
\beq
 P(W_l > s)
& \leq & \exp( - { s \over 4} \log ( 1+ { s \over \lambda_l}) ).
\eeq

For $t\geq y  > 5 \lambda_l$,
\beq
E(W_l - t )_+ & =  & \int_{t }^\infty P( W_l > s) ds \nn \\
& \leq & \int_{ t }^\infty
\exp( - { s \over 4} \log ( 1+ { s \over \lambda_l}) ) ds \nn \\
& \leq &  4  \exp(- { t \over 4} \log( 1+ t/\lambda_l)),
\eeq
\beq
\int_{y}^\infty { 1 \over t}  E(W_l - t )_+ dt
& \leq &  4  \int_{y}^\infty { 1 \over  t } \exp(- { t  \over 4} \log( 1+ t/\lambda_l)) dt \nn \\
& \leq & { 16 \over  y } \exp( - {y  \over 4} \log (1+y/\lambda_l)) .
\eeq
Combining these inequalities yields
\begin{equation}
EW I( W > 2y m ) \leq 8 m \exp( - { y \over 4} \log ( 1+ y/\lambda)) ( 1+ 4/y) .
  \lbl{h22-1}
\end{equation}
\end{proof}

\par

\bigskip

\noindent {\large\bf Acknowledgments.}
Louis Chen and Xiao Fang were partially supported by Grant C-389-000-010-101 at the National University of Singapore. Part of the revision was done when Xiao Fang was visiting Stanford University supported by NUS-Overseas Postdoctoral Fellowship from the National University of Singapore. Qi-Man Shao was partially supported by Hong Kong RGC GRF-602608, 603710 and CUHK2130344. The authors thank
    two referees for their valuable comments and suggestions that significantly improved the exposition of the paper.

\par

\bigskip

\noindent{\large\bf References}
\begin{description}
\item
Arratia, R., Goldstein, L. and Gordon, L. (1989).
Two moments suffice for Poisson approximations: the Chen-Stein method.
{\it Ann. Probab.} {\bf 17}, 9-25.
\item
Arratia, R., Goldstein, L. and Gordon, L. (1990).
Poisson approximation and the Chen-Stein method. With comments and a rejoinder by the authors.
{\it Statist. Sci.} {\bf 5}, 403-434.
\item
Baldi, P., Rinott, Y. and Stein, C. (1989).
A normal approximation for the number of local maxima of a random function on a graph.
{\it Probability, Statistics, and Mathematics}, 59--81. Academic Press, Boston, MA.
\item
Barbour, A. D. (1982).
Poisson convergence and random graphs.
{\it Math. Proc. Cambridge Philos. Soc.} {\bf 92}, 349-359.
\item
Barbour, A. D., Chen, L. H. Y. and Choi, K. P. (1995).
Poisson approximation for unbounded functions, I: Independent summands.
{\it Statist. Sinica} {\bf 2}, 749-766.
\item
Barbour, A. D., Holst, L. and Janson, S. (1992).
{\it Poisson Approximation}. Oxford Science Publications, Oxford.
\item
Chatterjee, S., Diaconis, P. and Meckes, E. (2005).
Exchangeable pairs and Poisson approximation.
{\it Probab. Surv.} {\bf 2}, 64-106.
\item
Chen, L. H. Y. (1975).
Poisson approximation for dependent trials.
{\it Ann. Probab.} {\bf 3}, 534-545.
\item
Chen, L. H. Y. and Choi, K. P. (1992).
Some asymptotic and large deviation results in Poisson approximation.
{\it Ann. Probab.} {\bf 20}, 1867-1876.
\item
Goldstein, L. and Rinott, Y. (1996).
Multivariate normal approximations by Stein's method and size bias couplings.
{\it J. Appl. Probab.} {\bf 33}, 1-17.
\item
Shao, Q. M. and Zhou, W. X. (2012).
Cram\'er type moderate deviation theorems for self-normalized processes.
{\it Preprint.}
\item
Stein, C. (1986).
{\it Approximate Computation of Expectations.}
Institute of Mathematical Statistics, Hayward, CA,
\end{description}

\vskip .65cm
\noindent
Department of Mathematics,
National University of Singapore,
10 Lower Kent Ridge Road,
Singapore 119076,
Republic of Singapore.
\vskip 2pt
\noindent
E-mail: (matchyl@nus.edu.sg)
\vskip 2pt
\noindent
Department of Statistics and Applied Probability,
National University of Singapore,
6 Science Drive 2,
Singapore 117546,
Republic of Singapore,

\noindent
and Department of Statistics,
Sequoia Hall, 390 Serra Mall, Stanford University, Stanford, CA 94305-4065, USA.
\vskip 2pt
\noindent
E-mail: (stafx@nus.edu.sg)
\vskip 2pt
\noindent
Department of Statistics,
The Chinese University of Hong Kong, Shatin, N.T., Hong Kong,
P.R. China.
\vskip 2pt
\noindent
E-mail: (qmshao@cuhk.edu.hk)
\vskip .3cm
\end{document}